%% file: main.tex
\documentclass[a4paper,12pt,leqno]{amsart}
\usepackage[colorlinks=true,linkcolor=darkblue,citecolor=darkblue]{hyperref}

\input{Preamble}

\definecolor{darkblue}{rgb}{0.0, 0.0, 0.55}
\usepackage{verbatim}

\begin{document}

\title[]{A spine for the decorated Teichmüller space of a punctured non-orientable surface}

\author[N. Colin]{Nestor Colin}
\author[R. Jiménez Rolland]{Rita Jiménez Rolland}
\author[P. L. León Álvarez ]{Porfirio L. León Álvarez}
\address{Instituto de Matemáticas, Universidad Nacional Autónoma de México. Oaxaca de Juárez, Oaxaca, México 68000}
\email{rita@im.unam.mx}
\email{ncolin@im.unam.mx}
\email{porfirio.leon@im.unam.mx}

\author[L. J. Sánchez Saldaña]{Luis Jorge S\'anchez Salda\~na}
\address{Departamento de Matemáticas, Facultad de Ciencias, Universidad Nacional Autónoma de México}
\email{luisjorge@ciencias.unam.mx}


\date{}



\begin{abstract}
Building on work of Harer \cite{Ha86},  we construct a {\it spine} for the decorated Teichm\"uller space of a non-orientable surface with at least one puncture and negative Euler characteristic. We compute its dimension, and show that the deformation retraction onto this spine is equivariant with respect to the pure mapping class group of the non-orientable surface.  As a consequence, we obtain a model for the classifying space for proper actions of the pure mapping class group of a punctured non-orientable surface, which is of minimal dimension in the case there is a single puncture.
\end{abstract}

\maketitle


\section{Introduction}

Let \( \Sigma_g \) be a possibly non-orientable closed connected surface of genus \( g \) and consider a collection \( \{p_1, \ldots, p_s\} \) of distinguished distinct points in \( \Sigma_g \). The {\it mapping class group} $\modgs{\Sigma}$ is the group of isotopy classes of all  (just orientation-preserving if $\Sigma_g$ is orientable) diffeomorphisms $f\colon \Sigma_g^s\to \Sigma_g^s$ where $\Sigma_g^s:=\Sigma_g - \{ p_1, \ldots, p_s \}$. The group $\modgs{\Sigma}$ permutes the  set $\{p_1, \ldots, p_s\}$ and the kernel of such action  is the {\it pure mapping class group} $\pmodgs{\Sigma}$. These groups are related by the following short exact sequence
$$1\rightarrow \pmodgs{\Sigma}\rightarrow \modgs{\Sigma}\rightarrow \mathfrak{S}_s\rightarrow 1,$$
where $\mathfrak{S}_s$ denotes the symmetric group on $s$ letters. The group $\modgs{\Sigma}$ is virtually torsion free, and its virtual cohomological dimension ($\vcd$) was computed by Harer \cite[Theorem 4.1]{Ha86} for orientable surfaces, and by Ivanov \cite[Theorem~6.9]{I87} for non-orientable surfaces. We use the notation $F_g$  if the surface $\Sigma_g$ is orientable and $N_g$ if it is non-orientable. The group $ \modgs{F} $ is an index $2$ subgroup of the {\it extended mapping class group} $\Mod^{\pm}( F_{g}^{s} ) $ of isotopy classes of all diffeomorphisms of $F_g^s$.

Take $\Delta =\{p_1, \ldots, p_m\}$ a subset of {\it marked points} in $F_g$, where $1\leq m\leq s$.  In \cite{Ha86} Harer explicitly described  a cell complex $\mathcal{Y}=\mathcal{Y}_g^{s,m}$  inside the {\it decorated Teichmüller space} $\T(F_g^s;\Delta)$ 
 onto which  $\T(F_g^s;\Delta)$ may be $\PMod(F_{g}^s)$-equivariantly retracted. The complex $\mathcal{Y}$ is often called {\it Harer's spine}, and when $m=1$ it gives a spine for the  Teichmüller space $\T(F_g^s)$ of equivalence classes of marked Riemann surfaces. { It is used in Harer's computation \cite[Theorem 4.1]{Ha86} of  $\vcd(\PMod_{g}^s)$,} and it plays a fundamental role in Harer--Zagier's computation \cite{HarerZagier} of the Euler characteristic of the moduli space of curves. Moreover, it gives a model  for $\underline E \PMod(F_g^s)$, the {\it classifying space for proper actions of $\PMod(F_g^s)$}, which is of minimal dimension when $s=1$; see for instance \cite[Introduction]{CJRLASS}. 

 In this paper, we consider a non-orientable surface $N_{g+1}^n$ of genus $g+1$ with $n\geq 1$ punctures. We take $1\leq \ell\leq n$,  $s=2n$, $m=2\ell$,   and  the {\it orientable double cover} $\pi:F_g\rightarrow N_{g+1}$,  where the group of covering transformations is generated by an orientation reversing diffeomorphism $\sigma:F_g\rightarrow F_g$ such that $\sigma(p_{2i} )= p_{2i-1}$ for all $1\leq i\leq n$. The set $\pi(\Delta)$ consists of  $\ell$ {\it marked points} in $N_{g+1}$. Building on the work of Harer, and by identifying the decorated Teichmüller space $\T(N_{g+1}^n;\pi(\Delta))$ of the non-orientable surface  with a subspace of the decorated Teichmüller space $\T(F_g^s;\Delta)$ of its orientable double cover, we construct a {\it spine} for $\T(N_{g+1}^n;\pi(\Delta))$: a cell complex $\mathcal{Z}=\mathcal{Z}_{g+1}^{n,\ell}$  inside the  decorated Teichmüller space $\T(N_{g+1}^n;\pi(\Delta))$ 
 onto which  $\T(N_{g+1}^n;\pi(\Delta))$ may be $\PMod(N_{g+1}^n)$-equivariantly retracted. Our main result is the following.

 \begin{theorem}\label{main:th:1}
Let $1 \leq \ell \leq n$  and $g +n> 1$. There exists  a $\PMod(N_{g+1}^n)$-equivariant spine $\calZ=\mathcal{Z}_{g+1}^{n,\ell}$ for the decorated Teichmüller space $\T(N_{g+1}^n;\pi(\Delta))$ of dimension 
\[
    \dim(\calZ) = \begin{cases}
    \vcd(\PMod(N_{g+1}^n))+\ell, & \text{if } \ell< n, \\
    \vcd(\PMod(N_{g+1}^n))+(\ell-1), & \text{if } \ell = n.
    \end{cases}
\] 
Moreover, the spine $\mathcal{Z}$ gives a model for the classifying space for proper actions of $\PMod(N_{g+1}^n)$. 
\end{theorem}

When $\ell=1$, the decorated Teichmüller space $\T(N_{g+1}^{n};\pi(\Delta))$ coincides with the Teichmüller space $\T(N_{g+1}^{n})$ of equivalence classes of marked Klein surfaces. In this case, our \cref{main:th:1} specializes to the following result.

\begin{corollary}[A spine for Teichm\"uller space of a non-orientable surface with punctures]\label{spine:one:puncture}
Let $1 \leq \ell \leq n$  and $g +n> 1$. There exists  a $\PMod(N_{g+1}^n)$-equivariant spine $\calZ=\mathcal{Z}_{g+1}^{n,1}$ for  Teichmüller space $\T(N_{g+1}^n)$ of dimension 
\[
    \dim(\calZ) = \begin{cases}
    \vcd(\PMod(N_{g+1}^n))+1, & \text{if }  n>1, \\
    \vcd(\PMod(N_{g+1}^n)), & \text{if } n=1.
    \end{cases}
\] 
Hence, $\calZ$ is a model for the classifying space for proper actions of $\PMod(N_{g+1}^n)$, which is of minimal dimension when $n=1$.
\end{corollary}

Recall that the \emph{proper geometric dimension} $\gdfin(G)$ of a group \( G \) is defined as the minimum integer \( n \) such that there is a model for \( \underline{E}G \) of dimension \( n \).  
Our \cref{spine:one:puncture}, combined with arguments in the same lines as \cite[Corollary 3.5]{CJRLASS}, computes the proper geometric dimension of $\PMod(N_{g+1}^n)$, answering a question posed by Hidber, Sánchez Saldaña and  Trujillo-Negrete \cite[Question 7.2]{MR4432529}.

\begin{corollary}\label{proper:dim:non:orien:mcgs:1}
For all $n\geq 1$ and $g+n>1$, the proper geometric dimension of $\PMod(N_{g+1}^n)$ is equal to its virtual cohomological dimension $\vcd(\PMod(N_{g+1}^n))$.
\end{corollary}

Our main result \cref{main:th:1} follows from \cref{restriction:Retraction}, which proves the existence of an equivariant spine; \cref{thm:main} where its  dimension is computed, and \cref{spine:model:proper:action} which proves that the spine gives a model for the classifying space for proper actions. 

\subsection*{Outline of the paper and the proof of Theorem \ref{main:th:1}}

In \cref{Sec:Preliminaries} we recall relevant definitions and introduce the notation that will be used through the paper. We first describe how the decorated Teichmüller of a non-orientable surface is related to the decorated Teichmüller of its orientable double cover. By \cite[Theorem 4.3]{Nestor:Xico}, the orientable double cover $\pi$ identifies the  Teichmüller space $\T(N_{g+1}^n)$ of a non-orientable surface  with the subspace of fixed points of the involution $\sigma$ acting on the Teichmüller space $\T(F_g^s)$ of its orientable double cover. We use this to show, in \cref{prop:NonOrientDecorated}, that the double cover $\pi$ also induces an equivariant homeomorphism between $\T(N_{g+1}^n;\pi(\Delta))$ and the subspace $\T(F_{g}^s;\Delta)^{\sigma}$ of $\sigma$-fixed points. 

To construct the spine $\calZ$, we first recall in \cref{Sec:Ideal:Triangulation} an ideal triangulation $\calA^0-\calA_\infty^0$ of the decorated Teichmüller space $\T(F_{g}^s;\Delta)$ introduced in \cite[Section 1]{Ha86}. It is defined as a subspace of $\calA^0$, the barycentric subdivision of a simplicial complex of arcs in the surface $F_g$. By considering the set of fixed points under the action of $\sigma$ on $\calA^0-\calA_\infty^0$, we obtain in \cref{restriction:harer:map} an ideal triangulation for $\T(N_{g+1}^n;\pi(\Delta))$. In \cref{Sec:SpineY} we review the definition of Harer's spine $\calY$ from \cite[Section 2]{Ha86}.  We observe that Harer's deformation retraction from $\T(F_{g}^s;\Delta)$ onto $\calY$ is actually equivariant  with respect to the subgroup of the extended mapping class group generated by $\sigma$ and the mapping classes that preserve the set $\Delta$. From this, we deduce in \cref{restriction:Retraction} that the set of $\sigma$-fixed points $\calY^{\sigma}$ gives a $\PMod(N_{g+1}^n)$-equivariant spine for $\T(N_{g+1}^n;\pi(\Delta))$.

We compute the dimension of this spine in \cref{Sec:Spine} to obtain \cref{thm:main}. This is done by following the approach in the proof of \cite[Theorem 1]{CJRLASS}, and by studying properties of $\sigma$-invariant arcs systems.  Finally, in \cref{teic:model:proper:action} we show how the decorated Teichmüller spaces, and the spine that we have constructed, give models for classifying spaces for proper actions of pure mapping class groups.

\subsection*{Acknowledgments.}   The first author was funded by SECIHTI through the program \textit{Estancias Posdoctorales por México.} The third author's work was supported by UNAM \textit{Posdoctoral Program (POSDOC)}. All authors are grateful for the financial support of DGAPA-UNAM grant PAPIIT IA106923.


\section{Decorated Teichmüller spaces and the orientable double cover}\label{Sec:Preliminaries}
In this section we introduce the notation that is used for the rest of the paper, and we describe how the decorated Teichmüller of a non-orientable surface is related to the decorated Teichmüller of its orientable double cover.  

Let \( \Sigma_g \) be a (possibly non-orientable) closed connected surface of genus \( g \), and  consider a set $\{p_1,\ldots, p_s \}$ of distinguished distinct points on $\Sigma_g$. We use the notation $N_g$  if the surface $\Sigma_g$ is non-orientable, and $F_g$ if it is orientable.  {In the latter case we take $F_g$ oriented. }

Let $\Diff(\Sigma_g^s)$ denote the group of diffeomorphisms of $\Sigma_g$ that preserve the distinguished points setwise, and let $\Diff_0(\Sigma_g^s)$ be the subgroup $\Diff(\Sigma_g)$ of all diffeomorphisms isotopic to the identity rel $\{p_i\}$. Note that the {\it extended mapping class group} $\Mod^{\pm}(\Sigma_g^s)$ is isomorphic to the quotient $\Diff(\Sigma_g^s)/ \Diff_0(\Sigma_g^s)$. When the surface $\Sigma_g$ is non-orientable we omit ${\pm}$ from the notation, and if $s=0$ we remove it as well.


\subsection{Decorated Teichmüller spaces for orientable and non-orientable surfaces}\label{Teichmüller:orientable:nonorientable}

The Teichmüller space of an orientable surface can be defined in several equivalent ways. We recall here the definition from \cite[Chapter 1]{Harer88Moduli} used in \cite{Ha86}. See also \cite[Sections 2 and 4]{Nestor:Xico}.

A \textit{marked Riemann surface} is a triple $(R,(q_1,\ldots,q_s),[f])$ where $R$ is a Riemann surface, $q_1,\ldots,q_s$ are distinct ordered points on $R$, and the \emph{marking} $[f]$ is the isotopy class of an orientation preserving homeomorphism $f: R \to F_g$ with $f(q_i)=p_i$ for all $i$. Here $[f]$  denotes the isotopy class of $f$ rel $\{q_i\}$.  We denote the set of marked Riemann surfaces by $\mathcal{M}(F_g^s)$.

There is an action of $\Diff(F_g^s)$ on $\mathcal{M}(F_g^s)$ given as follows:   
\begin{equation}\label{MCG_Action}
h\cdot(R,\{q_1,\ldots,q_s \},[f])=\begin{cases}
    (R,(q_{\eta(1)},\ldots,q_{\eta(s)}),[h\circ f])\text{\ \ if\ \ $h$\ \ {\it preserves} the orientation,}\\
     (R^*,(q_{\eta(1)},\ldots,q_{\eta(s)}),[h\circ f])\text{\ \ if\ \ $h$\ \ {\it reverses} the orientation,}\\
\end{cases}
\end{equation}
where $h\in\Diff(F_g^s)$, and  $\eta\in \mathfrak{S}_s$ is the inverse of the permutation given by the action of $h$ on the set of distinguished points $\{ p_1, \ldots, p_s\}$ on $F_g$. We denote by  $R^*$ the {\it mirror image} of the Riemann surface $R$, see for instance \cite[Section 1.3]{Strebel68}.  See also  \cite[Section II.3E]{AS60}, where $R^*$ is called the {\it conjugate} surface of $R$.

Two marked Riemann surfaces $(R_1,(q_1^1,\ldots,q_s^1),[f_1])$ and $(R_2,(q_1^2,\ldots,q_s^2),[f_2])$ are equivalent if there is a diffeomorphism $h\in\Diff_0(F_g^s)$, such that $(f_2^{-1}\circ h\circ f_1):R_1\to R_2$ is analytic.

The space $\mathcal{M}(F_g^s)/\Diff_0(F_g^s)$ of equivalence classes of marked Riemann surfaces, denoted by $\T(F_g^s)$,  is the \textit{Teichmüller space of an orientable surface of genus $g$ with $s$ distingushed points}. It is known that $\T(F_g^s)$ is  diffeomorphic to $\mathbb{R}^{6g-6+2s}$; see for instance \cite[Proposition 7.1.1 $\&$ Theorem 7.6.3]{Hubb06}. Furthermore, the action of $\Diff(F_g^s)$ on $\mathcal{M}(F_g^s)$, induces an action of the extended mapping class group $\modgse$  on $\T(F_g^s)$.

On the other hand, every non-orientable surface $N_g$ can be endowed with a \textit{Klein surface structure} (i.e. a {\it dianalytic} structure), and the definition of Teichmüller space naturally extends to the non-orientable setting; see also \cite[Section 2.1]{HarerGJ01} and \cite[Section 4]{Nestor:Xico}.  A \emph{marked Klein surface} is a triple $(R,( q_1,\ldots,q_s),[f])$, where $R$ is now a Klein surface and the marking $[f]$ is the isotopy of a homeomorphism $f: R \to N_g$ with $f(q_i)=p_i$ for all $i$. We define the equivalence of two marked Klein surfaces in a similar way than for marked Riemann surfaces, but we now require that  $f_2^{-1}\circ h\circ f_1$ is a \emph{dianalitic} homeomorphism. 

The \textit{Teichmüller space $\T(N_g^s)$ of a non-orientable surface of genus $g$ with $s$ distingushed points} is the space $\mathcal{M}(N_g^s)/\Diff_0(N_g^s)$ of equivalence classes of marked Klein surfaces. Analogously as in the orientable setting, the mapping class group $\Mod(N_g^s)$ acts on $\T(N_g^s)$, and  the space $\T(N_g^s)$ is diffeomorphic to $\mathbb{R}^{3g - 6 + 2s}$; see \cite[Theorem 5.11.4]{Sep92} and \cite[Theorem 2.2]{PapPenner16}. 

\begin{remark}
    To simplify the exposition we follow the notation used in \cite{Ha86} and refer to the equivalence class of a marked Riemann  or Klein surface $(R, (q_1, \ldots, q_s), [f])$ simply as $R$.
\end{remark}

We now let \( 1\leq m \leq s \), and  from the collection \( \{p_1, \ldots, p_s\} \) of distinguished points in \( \Sigma_g \), take the subset \( \Delta = \{p_1, \ldots, p_m\} \) of \emph{marked points}.  As before, we let  $\Sigma_g^s:=\Sigma_g-\{p_1,\ldots,p_s\}$, and  take  $\Sigma_0:=\Sigma_g-P$, where $P=\{p_{m+1},\ldots,p_s\}$.

\begin{definition}
The \emph{decorated Teichm\"uller space} $\T(\Sigma_g^s;\Delta)$ is the space of all pairs $(R,\lambda)$, where $R\in\T(\Sigma_g^s)$ and {$\lambda=[\lambda_1:\lambda_2:\ldots:\lambda_m]$} is a projective class of a collection of positive weights on the $m$ points of $\Delta$. 
\end{definition}

Topologically $\T(\Sigma_g^s;\Delta)$ is homeomorphic to the product of $\T(\Sigma_g^s)$ and an open simplex {${\rm{K}}^{m-1}$} of dimension $m-1$. Therefore, $\dim \T(\Sigma_g^s;\Delta)=\dim\T(\Sigma_g^s)+m-1$. In particular,  $\T(\Sigma_g^s;\Delta)$ is just the Teichmüller space $\T(\Sigma_g^s)$ when $m=1$. 

\begin{notation}
We denote by $\mode(\Sigma_0,\Delta)$ the group of isotopy classes of all diffeomorphisms $f\colon \Sigma_0 \to \Sigma_0$ that preserve the set $\Delta$ of marked points in $\Sigma_0$. Notice that $\mode(\Sigma_0,\Delta)$ is isomorphic to a subgroup of $\Mod^{\pm}(\Sigma_{g}^s)$ that contains $\PMod^{\pm}(\Sigma_{g}^s)$. When $\Sigma_g$ is non-orientable, we remove $\pm$ from the notation.
\end{notation} 

There is a natural diagonal action of $\mode(\Sigma_0,\Delta)$ on $\T(\Sigma_g^s;\Delta)$.  It is given by formula (\ref{MCG_Action}) on the factor $\T(\Sigma_g^s)$. In the open simplex  ${\rm{K}}^{m-1}$, this action corresponds to a permutation on the set of marked points of $\Delta$ and their associated positive weights, induced by the mapping class. It is worth pointing out that $\T(\Sigma_g^s;\Delta)$ admits an action of $\Mod^{\pm}(\Sigma_{g}^s)$ only when $s=m$.



\subsection{The orientable double cover}\label{Orientable:Double:Cover} 
 The topological \emph{orientable double cover} of the non-orientable compact surface $N_{g+1}$ can be constructed (up to isomorphism) as follows. Consider a closed orientable surface $F_g$ embedded in $\R^3$ such that $F_g$ is invariant under reflections in the $xy$-, $yz$-, and $xz$- planes. {We take the orientation in $F_g$ induced by the embedding}. Let $\sigma :F_g \to F_g$ be the orientation reversing diffeomorphism  
$$  
    \sigma(x, y, z)  =  (-x, -y, -z) .
$$
Then the quotient  $F_g  / \langle \sigma\rangle$ is homeomorphic to $N_{g+1}$ and the natural projection  $\pi \colon F_g  \to N_{g+1}$ is a double cover of $N_{g+1}$ such that $\sigma$ is a covering transformation.

\begin{notation}\label{MAINotation}  Let $1 \leq \ell \leq n$, and take $s:=2n$ and $m:=2\ell$ throughout this paper to make our notation coherent with the previous section. Consider $\tilde X:=\{ p_1,p_2,\ldots, p_{2n} \}$ a collection of $s$ distinct distinguished points in $F_g$ such that $\sigma(p_{2i} )= p_{2i-1}$ for all $1\leq i\leq n$. Then $\pi(p_{2i})=\pi(p_{2i-1})$ for all $1\leq i\leq n$, and  $X:=\pi(\tilde X)$ gives a  collection of $n$ distinguished distinct points in $N_{g+1}$. Let $F_{g}^{s}:=F_g- \tilde{X}$ and $N_{g+1}^{n}:=N_{g+1}-{X}$.  Then the projection $\pi$ restricts to $F_{g}^{s}$, and gives the orientable double cover of $N_{g+1}^{n}$, which we denote again by $\pi:F_{g}^{s}\to N_{g+1}^{n}$.

\begin{figure}[!ht]
    \centering
    \begin{subfigure}[b]{0.7\textwidth}
        \centering
        \includegraphics[width=\textwidth]{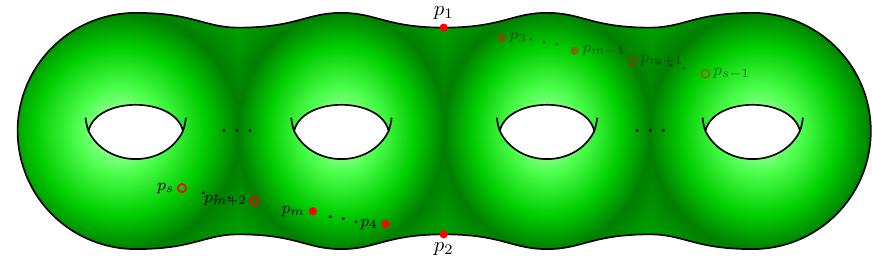} 
        \caption{Model surface in the case of even genus}
        \label{subfig:Model:Surface:Even}
    \end{subfigure}
    \hfill
    \begin{subfigure}[b]{0.8\textwidth}
        \centering
        \includegraphics[width=\textwidth]{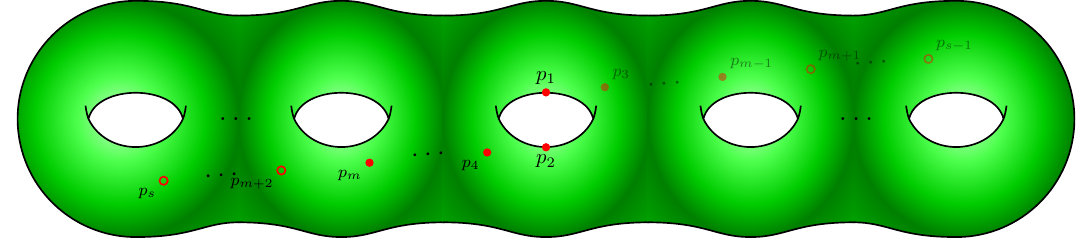} 
        \caption{Model surface in the case of odd genus}
        \label{subfig:Model:Surface:Odd}
    \end{subfigure}
    {\centering \caption{Model surface $F_0=F_g-P$ with sets of {\it punctures} $P$ and of {\it marked points} $\Delta$} \label{fig:Model:Surface}}
\end{figure}

Take a subset $\Delta=\{p_1,\ldots p_m\}$ of $\tilde{X}$.  Then $\Delta$ is a set of $m$  {\it marked points} in $F_0=F-P$, where  $P = \{ p_{m+1}, \ldots, p_s \}$.  Moreover, $\pi(\Delta)$ gives us a set  of $\ell$ {\it marked points} in $N_0=N_{g+1}-\pi(P)$.  An explicit model of the surface $F_0$ is shown in \cref{fig:Model:Surface}.          
\end{notation}



\subsection{Mapping class groups, decorated Teichmüller spaces and the orientable double cover}\label{Embed:Teich:Non:Orient} 

The orientable double cover $\pi$ induces a group homomorphism 
\[ 
    \Theta  \colon  \modgsn   \to  \modgso, 
\] 
 which is injective for $g \geqslant 2$ if $n=0$  and for all $g \geqslant 0$ if $n\geqslant 1$; see \cite[Theorem 3.2]{Nestor:Xico} and references therein. Furthermore, notice that $\sigma:F_g^s\rightarrow F_g^s$ represents a mapping class in $\Mod^{\pm}(F_{g}^{s})$ and via the monomorphism $\Theta$ we have
\[
    \modgsn\cong \modgso \cap C_{\Mod^{\pm}(F_{g}^{s})}(\sigma),
\]
where $C_{\Mod^{\pm}(F_{g}^s)}(\sigma)$ denotes the centralizer in $\Mod^{\pm}(F_{g}^{s})$ of the mapping class represented by the diffeomorphism $\sigma$.  Therefore, under the monomorphism $\Theta$, the groups $\PMod(N_{g+1}^{n})$ and $\Mod(N_{g+1}^{n})$ can be identified as subgroups of $C_{\Mod^{\pm}(F_{g}^{s})}(\sigma)$.

Furthermore, since $\Delta$ and $P$ are $\sigma$-invariant, $\sigma$ also represents a mapping class in $\Mod^{\pm}(F_0,\Delta)$ and  $\Theta$ restricts to a monomorphism from $\Mod(N_0,\pi(\Delta))$ to $\Mod^{\pm}(F_0,\Delta)$. Moreover, 
\[
    \Mod(N_0,\pi(\Delta))\cong \Mod(F_0,\Delta)\cap C_{\Mod^{\pm}(F_0,\Delta)}(\sigma),
\]
where $C_{\Mod^{\pm}(F_0,\Delta)}(\sigma)$ is the centralizer in $\Mod^{\pm}(F_0,\Delta)$ of the mapping class represented by $\sigma$.

The orientable double cover $\pi$ also induces a topological embedding
  \[  \pi^*\colon \T(N_{g+1}^{n})\to \T(F_{g}^{s}).
\]
The function $\pi^*$ is actually an isometric embedding with respect to the corresponding Teichmüller metrics; see for example \cite[Section 4.5]{Sep92} and \cite[Section 4]{Nestor:Xico}. Furthermore, \cite[Theorem 4.5.1]{Sep92} and \cite[Theorem 4.3]{Nestor:Xico} show that the image of $\pi^*$ is given by  
 $$\pi^*(\T(N_{g+1}^{n}))=\T (F_{g}^{s})^{\sigma}$$
where $\T (F_{g}^{s})^{\sigma}:=\{ R\in\T (F_{g}^{s}): \sigma\cdot R = R\}$ is the set of fixed points of the action of (the mapping class represented by) $\sigma$ on $\T(F_{g}^{s})$. 

The embedding $\pi^*\colon \T(N_{g+1}^{n})\to \T(F_{g}^{s})$ and the monomorphism $\Theta  \colon  \modgsn   \to  \modgso$ are compatible with the actions of the mapping class groups on the
corresponding Teichmüller spaces: for any $g\in\modgsn$, we have $\pi^*(g\cdot R)=\Theta(g)\cdot \pi^*(R)$.  

\begin{proposition}\label{prop:FixNonOrient}
The Teichmüller space $\T(N_{g+1}^n)$ of a non-orientable surface  is  $\Mod(N_{g+1}^n)$-equivariantly homeomorphic to the  subspace $\T (F_{g}^{s})^{\sigma}$ of the Teichmüller space $\T(F_{g}^{s})$ of its orientable double cover.
\end{proposition}   

From our discussion above, it follows that the double cover $\pi$ induces a topological embedding  
$$\T(N_{g+1}^n;\pi(\Delta))\rightarrow\T(F_{g}^s;\Delta)\text{ given by }(R,{[\lambda_1:\ldots:\lambda_\ell]})\mapsto (\pi^*(R), {[\lambda_1:\lambda_1:\ldots:\lambda_\ell:\lambda_\ell])},$$ which is compatible with the monomorphism $ \Mod(N_0,\pi(\Delta))\to\Mod^{\pm}(F_0,\Delta)$ and the corresponding actions on the decorated Teichmüller spaces. Moreover, the analogous statement corresponding to \cref{prop:FixNonOrient} also holds. We denote by $\T(F_g^s;\Delta)^\sigma$  the subspace of $\T(F_g^s; \Delta)$ of fixed points of $\sigma$. 

\begin{proposition}\label{prop:NonOrientDecorated}
 The decorated Teichmüller space $\T(N_{g+1}^n;\pi(\Delta))$ of a non-orientable surface is $\Mod(N_0,\pi(\Delta))$-equivariantly homeomorphic to the subspace $\T(F_g^s;\Delta)^\sigma$ of the decorated Teichmüller space $\T(F_g^s; \Delta)$ of its orientable double cover.
\end{proposition}


\section{Ideal triangulations and spines for decorated Teichmüller spaces}\label{Sec:Ideal:Triangulation}

In this section we recall an ideal triangulation of the decorated Teichmüller space $\T(F_g^s;\Delta)$ introduced in \cite{Ha86} and Harer's construction of a spine for $\T(F_g^s;\Delta)$ from this triangulation.   We show in \cref{restriction:harer:map} and  \cref{restriction:Retraction} how to obtain from Harer's setting an ideal triangulation and a spine for the decorated Teichmüller space $\T(N_{g+1}^n; \pi(\Delta))$,  this is done by identifying this space with the subspace $\T(F_{g}^s; \Delta)^{\sigma}$ using the orientable double cover as described in \cref{Orientable:Double:Cover}.


\subsection{Arc systems and Harer's complex of arcs}

Let us recall the definition of Harer's complex of arcs as introduced in \cite[Section 1]{Ha86}.

Consider the closed orientable  surface $F_g$, and the sets of points $\Delta$ and $P$ as defined in \cref{Teichmüller:orientable:nonorientable}. As before let $F_0=F_g-P$.
A properly embedded path in \( F_0 \) between two points of \( \Delta \) 
will be called a \( \Delta \)-arc. The isotopy class in \( F_0 \) (rel \( \Delta \)) \( [\alpha_0, \ldots, \alpha_k] \) of a family of \( \Delta \)-arcs will be called a {\it rank-\( k \) arc system} if:

\begin{enumerate}
\item \( \alpha_i \cap \alpha_j \subset \Delta \) for distinct \( i \) and \( j \), and
 \item for each connected component \( B \) of the surface obtained by \textit{ splitting \( F_0 \) along \( \alpha_0, \ldots, \alpha_k \)}, the Euler characteristic of the double of \( B \) along \( \partial B - \Delta \) is negative. 
\end{enumerate}

Condition (2) ensures that $\Delta$-arcs are not null-homotopic (rel $\Delta$), and no two distinct $\Delta$-arcs are homotopic (rel $\Delta$).

\begin{remark} 
    Let us be more precise about the \textit{connected components} appearing in condition (2). Consider a collection of $\Delta$-arcs \( \alpha_0, \ldots, \alpha_k \) such that  condition (1) holds. Let $B^o$ be one of the connected components of the open surface $F_0-\cup_{i=0}^k \alpha_i$. Hence $B^o$ is the interior of a surface $B$ with non-empty boundary, and we have an \textit{attaching map} $\phi_B\colon \partial B\to \cup_{i=0}^k \alpha_i$. Since $\cup_{i=0}^k \alpha_i$ is canonically a 1-dimensional CW-complex (with 0-skeleton contained in $\Delta$), we can pull-back that cellular structure to $\partial B$. By an abuse of notation we call $\Delta$ the $0$-skeleton of $\partial B$. Whenever $B$ is homeomorphic to a disk or a punctured disk, we consider it as a polygon or a punctured polygon respectively, and every edge will be labeled with the corresponding $\Delta$-arc in the system. 
\end{remark}

\begin{definition}[Harer's complex of arcs] Let \(\calA=\calA(\Delta)\) be the  simplicial complex that has  a \(k\)-simplex \(\langle \alpha_0, \ldots, \alpha_k\rangle\) for each rank-\(k\) arc system in \(F_0\) and such that \(\langle\beta_0, \ldots, \beta_l\rangle\) is  a face of \(\langle\alpha_0, \ldots, \alpha_k\rangle\) if \(\{[\beta_0], \ldots, [\beta_l]\}\) is contained in \(\{[\alpha_0], \ldots, [\alpha_k]\}\). 
\end{definition}

The dimension of the complex $\mathcal{A}(\Delta)$ is $6g-7+2s+m$; see for example  \cite[Theorem 2.2]{CJRLASS} for an explicit computation. The  natural action of the group $\modeFD$ on $\Delta$-arcs induces an action on $\calA(\Delta)$ by simplicial automorphisms.

We say an arc system $\asA$ \emph{fills up} $F_0$  provided every point in $\Delta$ is the initial or final point of an arc in $\asA$, and each connected  component $B$ of the splitting of $F_0$ along $\asA$ is either homeomorphic to a disk or to a once-puntured disk.  Let $\calA_\infty(\Delta)$ be the codimension $2$ subcomplex of $\calA(\Delta)$ formed by all the arc systems that do not fill up $F_0$.  Notice that the action of $\modeFD$ on $\calA(\Delta)$ preserves the subcomplex $\calA_\infty(\Delta)$.

\begin{remark}
   The action of $\modeFD$ on $\mathcal{A}(\Delta)$ has inversions. For instance,  the involution $\sigma$ acts as an inversion in the $1$-simplex $\langle \alpha_0 ,\sigma(\alpha_0) \rangle $, where $\alpha_0$ is the $\Delta$-arc  in \cref{fig:MaxAS:A}. To avoid this we work  with the barycentric subdivision.  
\end{remark}

Let us denote by $\calA^0$ and $\calA_\infty^{0}$ the barycentric subdivision of the complexes $\calA(\Delta)$ and $\calA_\infty(\Delta)$, respectively.  More explicitly, $\calA^0$ is the simplicial complex with vertex set given by the set of all arc systems in $F_0$, and a $k$-simplex is given by a chain $\asA_0 \subsetneq \cdots \subsetneq \asA_k$ of arc systems. Notice that the group $\modeFD$ acts simplicially  without inversions on $\calA^0$, and it restricts to an action of $\modeFD$ on $\calA^0-\calA_\infty^0$.


\subsection{Ideal triangulations of decorated Teichmüller spaces.} Let $s\geq 1$ and $2g+s>2$.
The following result states that $\mathcal{A}(\Delta) - \mathcal{A}_{\infty}(\Delta)$ defines an \textit{ideal} triangulation of the decorated Teichmüller space $\T(F_g^s;\Delta)$.

\begin{theorem}\cite[Theorem 1.3]{Ha86}\label{Theo:HomeoTeich}
There exists a $\Mod^{\pm}(F_0;\Delta)$-equivariant homeomorphism
\[
    \phi\colon \T(F_g^s;\Delta) \longrightarrow \mathcal{A}(\Delta) - \mathcal{A}_{\infty}(\Delta).
\]
\end{theorem}

Using the hyperbolic geometry perspective, Penner gave an alternative construction of this ideal triangulation in \cite{Penner87} in the case $s=m$.

\begin{remark} Harer's result \cite[Theorem 1.3]{Ha86} actually states that $\phi$ is a $\PMod(F_g^s)$-equivariant homeomorphism. We briefly recall Harer's definition of $\phi$, and point out that this homeomorphism is in fact $\modeFD$-equivariant.  See also \cite[Chapter 2]{Harer88Moduli} for $m=s=1$ and \cite[Section 2.4]{HarerGJ01} and \cite[Section 5.3]{Mondello11} for $m=s$. 

 Let $(R,\lambda)\in \T(F_g^s;\Delta)$. Then $R$ is a marked Riemann surface with distinguished points $q_1,\ldots, q_s$ and a marking $f:R\rightarrow F_g$, and $\lambda=[\lambda_1:\ldots: \lambda_m]$ is a projective class of positive weights on the $m$ points of $\Delta_R:=f^{-1}(\Delta)=\{q_1,\ldots,q_m \}$.  By a result of Strebel  \cite[Part 9,10 and 11]{Strebel68}, \cite[Theorem 5.2]{Mondello11}, there exists a unique (up to scalar multiplication) quadratic differential $\omega_R=\{\omega_{\alpha}(z_{\alpha})dz_{\alpha}^2\}$ on $R$ that has double poles 
in each point of $\Delta_R$ and no other poles, and such that the real trajectories of $\omega_{R}$ are closed. Moreover,  the singular trajectories of the differential $\omega_R$ decompose $R_0=R-\{q_{m+1},\ldots,q_s\}$ into $m$ disks $G_1,\ldots,G_m$, where each disk  $G_i$ contains a point $q_i\in\Delta_R$, and the projective class of the radii of $\{G_i\}$ is equal to $\lambda$.
The imaginary trajectories of $\omega_R$ that begin and end at points in $\Delta_R$ break into parallel families. We select one leaf of each family and apply $f$ to obtain an arc system $[\alpha_0, \ldots, \alpha_k]$  in $F_g$. These arcs, combined with the projective class of weights determined by $\omega_R$, give a point $\phi((R,\lambda))$ in the interior of the simplex $\langle \alpha_0, \ldots, \alpha_k \rangle$ which lies in $\mathcal{A}(\Delta)-\mathcal{A}_{\infty}(\Delta)$.  The fact that the homeomorphism $\phi$ is equivariant under the action of $\Mod^{\pm}(F_0;\Delta)$ follows from noticing that $\bar{\omega}_R=\{\omega_{\alpha}(\bar{z}_{\alpha})d\bar{z}_{\alpha}^2\}$ is (up to scalar multiplication)  the Strebel quadratic differential $\omega_{R^*}$ on $R^*$, and hence it generates the same imaginary and real trajectories as $\omega_R$.
\end{remark}

Now we let $1\leq \ell\leq n$, $g+n>1$, and take $s:=2n$ and $m:=2\ell$. Consider \cref{MAINotation} and the orientable double cover $\pi: F_g^s\rightarrow N_{g+1}^n$ as described in \cref{{Orientable:Double:Cover}}. Notice that the homeomorphism $\phi$  from \cref{Theo:HomeoTeich} induces a   $C_{\modeFD}(\sigma)$-equivariant homeomorphism 
\[
\phi^\sigma \colon \T(F_g^s;\Delta)^{\sigma} \longrightarrow (\mathcal{A}^0 - \mathcal{A}_{\infty}^0)^{\sigma},
\] 
where $(\mathcal{A}^0 - \mathcal{A}_{\infty}^0)^{\sigma}$ es the subspace of $\mathcal{A}^0 - \mathcal{A}_{\infty}^0$ consisting of points fixed by $\sigma$.  It is worth noticing that $(\mathcal{A}^0 - \mathcal{A}_{\infty}^0)^{\sigma}$ is a simplicial complex since we are working in the baricentric subdivision of $\mathcal{A}$.

On the other hand, from \cref{Embed:Teich:Non:Orient} we know that the double cover $\pi$ identifies $\Mod(N_0,\pi(\Delta))$ with a subgroup of $C_{\Mod^{\pm}(F_0,\Delta)}(\sigma)$. Therefore, by combining $\phi^\sigma$ with \cref{prop:NonOrientDecorated}, we obtain an ideal triangulation $(\mathcal{A}^0 - \mathcal{A}_{\infty}^0)^{\sigma}$ for the decorated Teichmüller space of a non-orientable surface $N_{g+1}^n$.

\begin{corollary}\label{restriction:harer:map}
There is a $\Mod(N_0;\pi(\Delta))$-equivariant homeomorphism
\[
\varphi\colon \T(
N_{g+1}^n;\pi(\Delta)) \longrightarrow (\mathcal{A}^0 - \mathcal{A}_{\infty}^0)^{\sigma}.
\] 
\end{corollary}

In \cref{sigma:arc:simplex:dim} below we show that $(\mathcal{A}^0 - \mathcal{A}_{\infty}^0)^{\sigma}$ has dimension $3g+2n+\ell-4$.



\subsection{Spines for decorated Teichmüller spaces}\label{Sec:SpineY}

In \cite[Section 2]{Ha86}, Harer used the ideal triangulation $\calA^0-\calA_\infty^0$ of the decorated Teichmüller space $\T(F_g^s;\Delta)$ to construct a {\it spine}: a cell complex  inside $\T(F_g^s;\Delta)$ onto which $\T(F_g^s;\Delta)$ may be $\PMod(F_{g}^s)$-equivariantly {deformation} retracted. In this section, we recall Harer's construction of the spine $\mathcal{Y}$. By identifying once again $\T(N_{g+1}^n;\pi(\Delta))$ with the subspace $\T(F_{g}^s; \Delta)^{\sigma}$ via the orientable double cover,  we see in \cref{restriction:Retraction} that the $\sigma$-fixed points of $\calY$ give a spine for the decorated Teichmüller space $\T(N_{g+1}^n;\pi(\Delta))$.  In \cref{Sec:Spine} we compute its dimension.

\begin{definition}[Harer's spine] Let $\calY(=\calY_g^{s,m})$ be the subcomplex of $\calA^0$ spanned by the arc systems that fill up $F_0$. Therefore, the vertices of $\calY$ consist of the set of arc systems that fill up $F_0$ and a $k$-simplex of $\calY$ is given by a chain of arc systems
$\asA_0 \subsetneq \cdots \subsetneq \asA_k$,
where  all the $\asA_i$ fill up $F_0$.  We call {\it Harer's spine} both the complex $\calY$ and the subspace $\phi^{-1}(\calY)\subset \T(F_g^s;\Delta)$.
\end{definition}

\begin{remark}
    In \cite[Section 2]{Ha86}, Harer defines two complexes $Y$ and $Y^0$, the first one as the dual complex of $\calA$ and the second as a subcomplex of $\calA^0$. In our notation $\calY$ coincides with $Y^0$. Harer also points out that $Y^0$ is the baricentric subdivision of $Y$; hence the geometric realization of these complexes are homeomorphic.
\end{remark}

\begin{theorem}\cite[Theorem 2.1]{Ha86}\label{Theo:Retraction}
The complex $\mathcal{A}^0$ equivariantly deformation retracts onto the complex $\mathcal{Y}$. This retraction provides a $\modeFD$-equivariant homotopy equivalence between $\mathcal{A}^0 - \mathcal{A}_\infty^0$ and $\mathcal{Y}$.
\end{theorem}

\begin{remark} In \cite[Theorem 2.1]{Ha86} Harer only states the existence of a $\PMod(F_g^s)$-equivariant deformation retraction from $\calA^0-\calA_\infty^0$ onto $\calY$. However, from Harer's proof it follows that this deformation retraction is, in fact, $\modeFD$-equivariant.
Indeed, his proof is done inductively by collapsing the stars in $\calA^0$ of cells in $\calA_{\infty}^0$.
For each $k\geq 1$, starting with $\calA_{0}^0=\calA^0$,  Harer gives a $\Mod(F_0;\Delta)$-equivariant deformation retraction from $\calA_{k-1}^0$ onto $\calA_{k}^0$, the complex obtained from $\calA_{k-1}^0$ by removing the open stars of all vertices of $\calA_{\infty}^0$ of {\it weight} $k-1$ (i.e the ones given by arc systems $[\alpha_0,\ldots,\alpha_{k-1}]$ of rank $k-1$). This way, he obtains an equivariant deformation retraction from  $\calA^0$ onto $\calA_{\max}^0=\calY$, with $\max=\dim \calA_{\infty}^0+1$.
Since the action of $\sigma$ leaves invariant the set of vertices  of $\calA_{\infty}^0$ of {\it weight} $k-1$, then $\sigma$ leaves invariant the complex $\calA_k^{0}$, and the deformation retraction is in fact $\sigma$-equivariant.
 \end{remark}

Let  $\calY^{\sigma}$ the subcomplex of $\calY$ of $\sigma$-fixed points. Notice that $\calY^{\sigma}$ is a subcomplex of $(\mathcal{A}^0)^\sigma$, and  since $\calY\subset \calA^0-\calA^0_{\infty}$, then $(\mathcal{Y})^\sigma$ is actually subspace of $(\calA^0-\calA^0_{\infty})^\sigma$.   By taking $\sigma$-fixed points, the $\modeFD$-equivariant deformation 
retraction from \cref{Theo:Retraction} restricts to $\calY^{\sigma}$ and combined with \cref{restriction:harer:map} gives the following result.

\begin{corollary}\label{restriction:Retraction} 
The complex  $\mathcal{Y}^{\sigma}$ is a $\modeFD$-equivariant deformation retract of $(\mathcal{A}^0)^{\sigma}$. This retraction gives a  $C_{\Mod^{\pm}(F_0,\Delta)}(\sigma)$-equivariant homotopy equivalence between $(\mathcal{A}^0 - \mathcal{A}_\infty^0)^{\sigma}$ and $\mathcal{Y}^{\sigma}$. Furthermore, $\varphi^{-1}(\mathcal{Y}^{\sigma})$ is a $\Mod(N_{g+1};\pi(\Delta))$-equivariant deformation retraction of $\T(N_g^n;\pi(\Delta))$.
\end{corollary}

\begin{definition}[A {spine} for  $\T(N_{g+1}^n;{\pi(\Delta)})$]\label{spine:equi} 
We refer to both the complex $\calY^\sigma$ and the subspace $\calZ:=\varphi^{-1}(\calY^\sigma)$ of $\T(N_{g+1}^n;{\pi(\Delta)})$ as {\it Harer's spine for $\T(N_{g+1}^n;{\pi(\Delta)})$}.  When we want to emphasize the parameters of the non-orientable surface we use the notation $\calZ_{g+1}^{n,\ell}$.
\end{definition}



\section{The dimension of the spine} \label{Sec:Spine}

The goal of this section is to prove \cref{thm:main} which gives the dimension of the spine $\mathcal{Z}$, from \cref{spine:equi},  for the decorated Teichmüller space  of a punctured non-orientable surface. Since this spine is given by $\calY^{\sigma}$, the $\sigma$-fixed points of Harer's spine $\calY=\calY_{g}^{s,m}$,  we first study (in \cref{SigmaInvariant}) arc systems that are fixed by the action of $\sigma$.

\begin{definition}
    Let $ \asA$ be an arc system in $F_0$. We say that $\asA$ is \textit{$\sigma$-invariant arc system} if and only if $\sigma\cdot \asA = \asA$.  
\end{definition}

In order to obtain an upper bound for the dimension of any $k$-simplex in $\mathcal{Y}^\sigma$, we  establish upper and lower bounds for the rank of any $\sigma$-invariant arc system that fills up $F_0$;  see \cref{SigmaInvariantMaximal} and \cref{lem:lower:bound:MAS}.  To show equality, we construct an explicit $k$-simplex in $\mathcal{Y}^\sigma$ that attains the upper bound. We do this  in \cref{Sec:dim} and obtain \cref{thm:main}. 



\subsection{\texorpdfstring{Properties of $\sigma$-invariant arc systems}{Properties of sigma-invariant arc systems}}\label{SigmaInvariant}

Consider the complexes $(\mathcal{A}^0)^{\sigma}$ and $\calY^\sigma$ appearing in \cref{restriction:Retraction}. First notice that $(\mathcal{A}^0)^{\sigma}$ is the subcomplex of $\mathcal{A}^0$ spanned by $\sigma$-invariant arc systems, and  $\calY^\sigma$ is precisely the subcomplex  of $(\calA^0)^\sigma$ spanned by the  $\sigma$-invariant arc systems that fill up $F_0$. More precisely,  the vertices of $\calY^\sigma$ consist of the set of $\sigma$-invariant arc systems that fill up $F_0$ and a $k$-simplex of $\calY^\sigma$ is given by a chain of arc systems
$\asA_0 \subsetneq \cdots \subsetneq \asA_k$,
where each $\asA_i$ is a $\sigma$-invariant arc system that fills up $F_0$.

The following explicit $\sigma$-invariant arc system will be useful for our constructions below.

\begin{example}[$\sigma$-invariant arc system with $2g+2$ arcs] \label{ArSy:SigInv:Canonical}

Consider the model surface from \cref{fig:Model:Surface}. 
We begin by constructing a $\sigma$-invariant arc system for a surface of even genus.

\begin{figure}[!ht]
    \centering
    \begin{subfigure}[b]{0.70\textwidth}
        \centering
        \includegraphics[width=\textwidth]{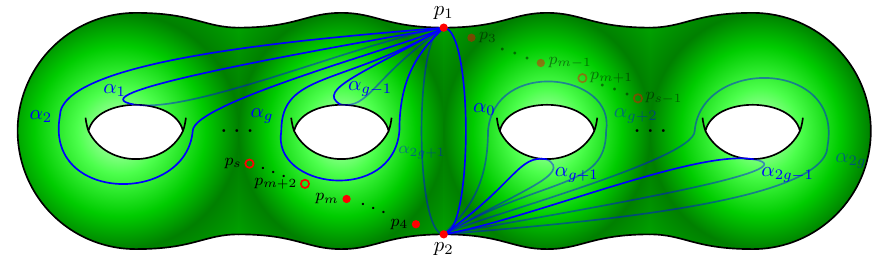} 
        \caption{$\sigma$-arc system of a surface of even genus}
        \label{subfig:MaxASEv:A}
    \end{subfigure}
    \hfill
    \begin{subfigure}[b]{0.85\textwidth}
        \centering
        \includegraphics[width=\textwidth]{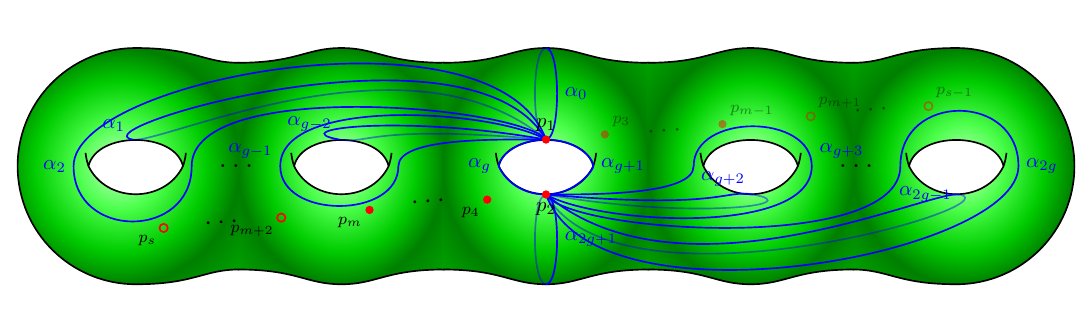} 
        \caption{$\sigma$-arc system of a surface of odd genus}
        \label{subfig:MaxASOd:A}
    \end{subfigure}
    \caption{A $\sigma$-invariant arc system $\asB$ with $2g+2$ arcs}
    \label{fig:MaxAS:A}
\end{figure}
Define two arcs, $\alpha_0$ and $\alpha_{2g+1} := \sigma(\alpha_0)$, which join the points $p_1$ and $p_2 = \sigma(p_1)$ in such a way that the surface splits into two regions, $F_1$ and $F_2$, with $\sigma(F_1) = F_2$. In $F_1$, select the loops $\alpha_1, \ldots, \alpha_g$ 
with base point $p_1$, each passing through one handle of the surface. Then, add the images {$\alpha_{g+1}:=\sigma(\alpha_{g-1}),\alpha_{g+2} := \sigma(\alpha_g), \ldots, \alpha_{2g-1} := \sigma(\alpha_1), \alpha_{2g} := \sigma(\alpha_2)$}. This set of arcs forms the arc system $\asB = [\alpha_0, \alpha_1, \ldots, \alpha_{2g+1} ]$, as illustrated in \cref{fig:MaxAS:A}. By construction, the arc system $\asB$ is $\sigma$-invariant. Cutting the surface $F_0$ along $\asB$, we obtain two polygons each with $2g+2$ sides with marked points and punctures in their interiors as show in \cref{fig:Pol:Sig:Arc:System}. 

An analogous $\sigma$-invariant arc system can be constructed for a surface of odd genus, as shown in \cref{subfig:MaxASOd:A}, yielding a similar representation of $F_0$ upon cutting along this arc system.

\begin{figure}[!ht]
    \centering
    \includegraphics[width=0.9\linewidth]{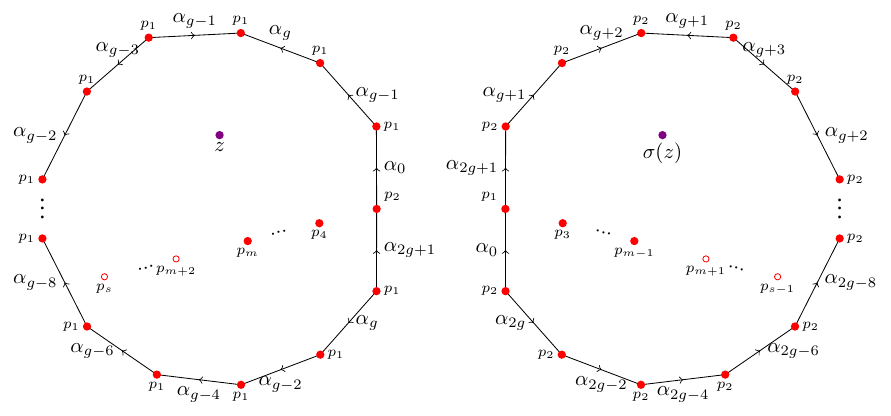}
    \caption{Polygons obtained by cutting $F_0$ along the arc system $\asB$.
    }
    \label{fig:Pol:Sig:Arc:System}
\end{figure}

\end{example}

\begin{lemma}\label{sigma:arcs:even:number}
Any $\sigma$-invariant arc system in $F_0$ has an even number of $\Delta$-arcs.
\end{lemma}
\begin{proof}
Let $\asA$ be a $\sigma$-invariant arc system in $F_0$. We claim that there is a filtration  of $\asA$
\[ 
    \asA_1\subsetneq \asA_2\subsetneq\cdots \subsetneq \asA_k =\asA,
\]
where each $\asA_i$ is a $\sigma$-invariant arc system in $F_0$ and $\vert \asA_i\vert=2i$ for all $i$.
We construct the filtration recursively. Let $\alpha_{s_1}\in \asA$, we define $\asA_1=\langle \alpha_{s_1}, \sigma(\alpha_{s_1})\rangle$, note that $\asA_1\subset \asA$ and $\vert \asA_1\vert=2$ this is because for all $\alpha\in \asA$ we have that $\sigma(\alpha)\in \asA$ and  they are not homotopic relative to $\Delta$. Indeed, if $\alpha_{s_1}$ is a loop, then by the construction of $F_0$, the initial point of $\alpha_{s_1}$ and $\sigma(\alpha_{s_1})$ are distinct, and the claim follows. Now, if $\alpha_{s_1}$ connects two distinct points of $\Delta$, then by the construction of $F_0$, the endpoints of $\alpha_{s_1}$ and $\sigma(\alpha_{s_1})$ are swapped, and hence they are not homotopic relative to $\Delta$.

We define $\asA_i$ from $\asA_{i-1}$ as follows: let $\alpha_{s_i}\in \asA-\asA_{i-1}$ then we define $\asA_i=\asA_{i-1}\cup \langle \alpha_{s_i}, \sigma(\alpha_{s_i}) \rangle$. Note that $\sigma(\alpha_{s_i})\neq \alpha_{k}$ for all $k\leq i$ and $\sigma(\alpha_{s_i})\neq \sigma(\alpha_k)$ for all $k<i$. In fact, if $\sigma(\alpha_{s_i})=\alpha_k$ for some $k<i$, then $\alpha_{s_i}=\sigma(\alpha_k)$ in particular $\alpha_{s_i}\in \asA_{i-1}$ and $\alpha_{s_i}\notin\asA- \asA_{i-1}$ and this contradicts our choice of $\alpha_{s_i}$. A similar argument gives us the other conclusion. Therefore it follows that $\vert \asA_i\vert=2i$ for all $i$.

Since $\vert \asA\vert<\infty$ we have that $\asA=\asA_{i}$ for some $i$, moreover by construction $\vert \asA_i\vert=2i$. In particular, $\asA$ has an even number of $\Delta$-arcs.
\end{proof}

\begin{lemma}\label{SigmaInvariantMaximal} There exist a $\sigma$-invariant arc system  $\asB_{max}$ of  maximal rank $6g-7+2s+m$.
\end{lemma}
\begin{proof}
According to \cite[Lemma 2.6]{CJRLASS}, the rank of any maximal arc system in $\mathcal{A}(\Delta)$ is $6g - 7 + 2s + m$.  Recall that an arc system is maximal if and only if it splits $F_0$ into triangles and once-punctured monogons, see \cite[Proposition 2.5]{CJRLASS}. Consider the $\sigma$-invariant arc system $\asB$ with $2g+2$ arcs described in \cref{ArSy:SigInv:Canonical} and the representation in \cref{fig:Pol:Sig:Arc:System}. By adding arcs to $\asB$,  we obtain a $\sigma$-invariant arc system $\asB_{max}$ of maximal rank. First, add the $2g-1$ diagonals in the left polygon and its $2g-1$ images under $\sigma$ of these arcs in the second polygon. Then, add three arcs for each of the marked point $p_3,\ldots, p_m$ yielding  $3(m-2)$ arcs. Finally, for each puncture in $F_0$ we add two arcs, so we get $2(s-m)$ additional arcs. We illustrate this construction for a surface of even genus in \cref{fig:MaxAS:B}. Notice that the total number of $\Delta$-arcs in $\asB_{max}$  is equal to 
\[
    (2g+2)+2(2g-1)+3(m-2)+2(s-m)=6g+2s+m-6=2(3g+2n+\ell-3).
\]
Therefore, by construction $\asB_{max}$ is a $\sigma$-invariant arc system of rank $6g+2s+m-7$. 
\begin{figure}[!ht]
    \centering
    \includegraphics[width=0.9\linewidth]{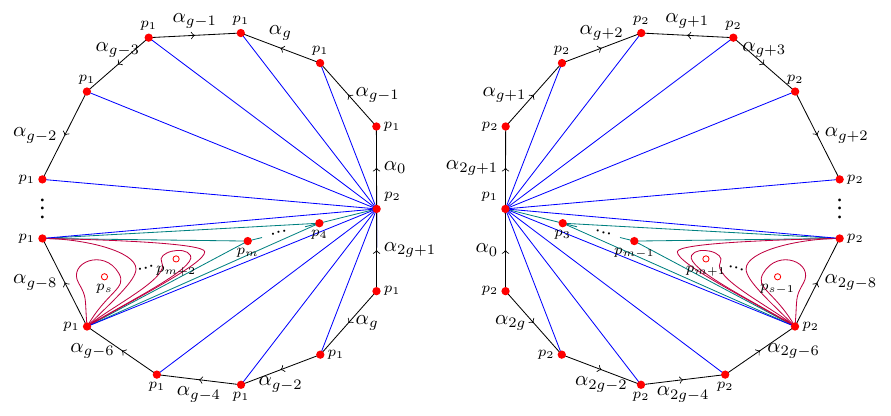} 
    \caption{A $\sigma$-invariant arc system $\asB_{max}$ of maximal rank for a surface of even genus}
    \label{fig:MaxAS:B}
\end{figure}
\end{proof}

Notice that the $\sigma$-invariant arc system $\asB_{max}$ defines a  chain of  maximal length
     \[
         \asB_0 \subsetneq \asB_1 \subsetneq \cdots \subsetneq \asB_{3g+2n+\ell-4} =\asB_{max},
     \]
 where $\asB_i= \{\alpha_0,\alpha_1, \ldots \alpha_{2i+1} \}$ is a $\sigma$-invariant arc system of rank $2i+1$, for each $i$. Furthermore, since the arc system $\asB_{max}$ fills up the surface $F_0$, this chain defines a simplex of maximal dimension in both $ (\calA^0)^\sigma$  and  $ (\calA^0-\calA^0_{\infty})^\sigma$.

\begin{corollary}\label{sigma:arc:simplex:dim} Both the simplicial complex $(\mathcal{A}^0)^\sigma$ spanned by $\sigma$-invariant arc systems and the ideal triangulation $(\mathcal{A}^0-\mathcal{A}^0_{\infty})^\sigma$ of $\T(
N_{g+1}^n;\pi(\Delta))$ have dimension $3g+2n+\ell-4$. 
\end{corollary}

\begin{lemma} \label{lem:lower:bound:MAS}
    Let $\asA$ be a $\sigma$-invariant arc system that fills up $F_0$. Then
    \[
        \rank(\asA) \geqslant \begin{cases}
        2g + s - 3 & \text{if } m < s, \\
        2g + s - 1 & \text{if } m = s.
        \end{cases} 
    \]
    Furthermore, there is an explicit $\sigma$-invariant arc system $\asB_{min}$ that fills up $F_0$ and attains this lower bound. 
\end{lemma}

\begin{proof}
    Let $\asA$ be a $\sigma$-invariant arc system that fills up $F_0$. In the proof of \cite[Theorem 1.1]{CJRLASS}, the authors show that the rank of an  arc system $\asA$ that fills up $F_0$ is bounded below by
\[
    \rank(\asA) \geq \begin{cases}
    2g + s - 3 & \text{if } m < s, \\
    2g + s - 2 & \text{if } m = s.
    \end{cases} 
\]
By \cref{sigma:arcs:even:number} we know that $\rank(\asA)$ must be odd. Hence, $\rank(\asA)\geq 2g + s - 1$ when $m=s$, as claimed. 

Now we construct a $\sigma$-invariant arc system that fills $F_0$ that attains the lower bound. We do it explicitly for a surface of even genus, the construction is similar for a surface of odd genus. 
\begin{figure}
    \centering
    \includegraphics[width=0.9\linewidth]{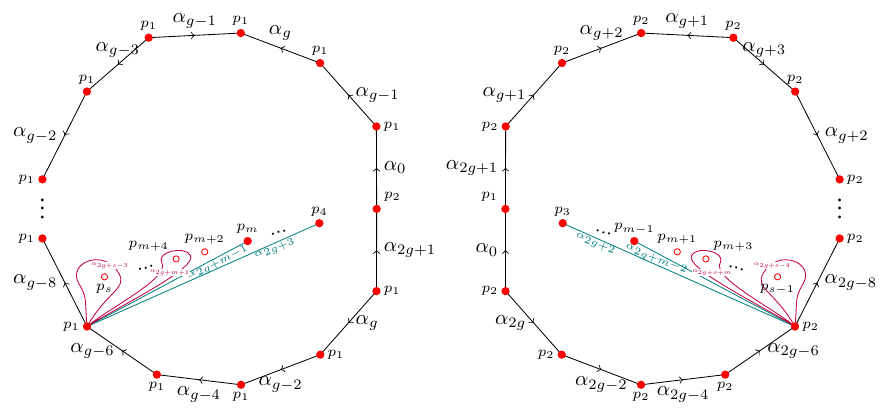}

    \caption{A minimal $\sigma$-invariant arc system $\asB_{min}$ that fills up $F_0$ of even genus}
    \label{fig:MinAS:Ev}
\end{figure}

Consider the $\sigma$-invariant arc system $\asB=[\alpha_0, \ldots, \alpha_{2g+1}]$ constructed in \cref{ArSy:SigInv:Canonical}.  Splitting the surface along this arc system yields two polygons as in \cref{fig:MaxAS:A}, each with $2g+2$ sides, containing marked points and punctures in their interiors. We add $m-2$ green arcs that connect each of the marked points in the interior of the polygons as indicated in \cref{fig:MinAS:Ev}, obtaining an arc system with $(2g+2)+(m-2)$ arcs. If $m=s$, then $F_0=F_g$ and we take $\asB_{min}$ to be this arc system. If $m<s$, we add $(s-m)-2$ purple loops that around each of the punctures of $F_0$ except two (one puncture in each polygon). We take $\asB_{min}$ as this arc system  with $(2g+2)+(m-2)+(s-m)-2$ arcs.  By construction, $\asB_{min}$ is $\sigma$-invariant and  fills up $F_0$.  It has rank equal to $ 2g+s-3$ if $m<s$ and rank equal to $2g+s-1$ if $m=s$.
\end{proof}



\subsection{Computation of the dimension}\label{Sec:dim}

The dimension of Harer's spine $\calY$ was explicitly computed in \cite[Theorem 1.1]{CJRLASS}. Here, we follow a similar strategy and use the results in \cref{SigmaInvariant} to compute the dimension of $\calY^\sigma$.

\begin{theorem}[Dimension of the spine $\mathcal{Z}$ for $\T(N_{g+1}^n;{\pi(\Delta)})$]\label{thm:main}
Let $1\leq \ell\leq n$ and take $g+n>1$.  The dimension of the spine $\calZ(=\calZ_{g+1}^{n,\ell})$  is given by
\[
    \dim(\calZ) = \begin{cases}
    2g + n+\ell- 2  & \text{if } \ell< n, \\
    2g + 2n - 3 & \text{if } \ell = n.
    \end{cases} 
\]  
\end{theorem}
\begin{proof} Let $m=2\ell$ and $s=2n$. Since $\calZ_{g+1}^{n,\ell}$ has the same dimension as $\calY^{\sigma}=(\calY_g^{s,m})^{\sigma}$, we work with the latter.
Recall that any $k$-dimensional simplex in $\calY^{\sigma}$ is determined by a chain 
$\asA_0 \subsetneq \asA_1 \subsetneq \dots \subsetneq \asA_k$, 
where each $\asA_i$ is a $\sigma$-invariant arc system that fills up $F_0$. By \cref{lem:lower:bound:MAS}, we have that 

 \[
        \rank(\asA_0) \geqslant \begin{cases}
        2g + s - 3 & \text{if } m < s, \\
        2g + s - 1 & \text{if } m = s.
        \end{cases} 
    \]
By  \cref{SigmaInvariantMaximal} it follows that $\rank(\asA_k) \leqslant 6g + 2s + m-7$. On the other hand,  by \cref{sigma:arcs:even:number}, each $\asA_i$ has an even number of arcs. Therefore,
\[
       k\leq \frac{1}{2}(\rank(\asA_k)-\rank(\asA_0))= \begin{cases}
        (4g+s+m-4)/2=2g+n+\ell-2 & \text{if } \ell < n, \\
        (4g+s+m-6)/2=2g+2n-3 & \text{if } \ell = n,
        \end{cases} 
    \]
which gives an upper bound for the dimension of $ (\calY)^{\sigma}$. 

To complete the proof, it suffices to construct an explicit $k$-simplex in $\calY^{\sigma}$ that realizes this upper bound. We exhibit this for a surface of even genus. The construction for a surface of odd genus is similar and we omit it.
   
Take $\asB_0$ to be $\asB_{min}$, the $\sigma$-invariant arc system that fills up $F_0$ constructed in \cref{lem:lower:bound:MAS} and illustrated in \cref{fig:MinAS:Ev}.  Then  $\asB_0$ has rank equal to $ 2g+s-3 $ if $m<s$ and equal to $2g+s-2$ if $m=s$.  To this arc system, we add a curve in the left polygon and its image under $\sigma$ in the right. We obtain a new $\sigma$-invariant arc system $\asB_1$ that fills up $F_0$. With the same procedure, by adding two curves at a time we obtain a chain
    \[
        \asB_{min}=\asB_0 \subsetneq \asB_1 \subsetneq \dots \subsetneq \asB_k=\asB_{max},
    \]
of $\sigma$-invariant arcs that fill up $F_0$.  We finish by obtaining $\asB_k$ as the maximal $\sigma$-invariant arc system $\asB_{max}$ from \cref{SigmaInvariantMaximal} illustrated in \cref{fig:MaxAS:B}, and $\rank(\asB_k)=6g+2s+m-7$. Thus $2k = \rank(\asB_k)-\rank(\asB_0)$, and the $k$-simplex formed by this chain has the desired dimension.
\end{proof}

In the case when $\ell=1$, we have that $|\pi(\Delta)|=1$,   and $\T(N_{g+1}^{n};\pi(\Delta))=\T(N_{g+1}^{n})$. 
Therefore, \cref{restriction:Retraction} and \cref{thm:main}  give the following result.

\begin{corollary}[A spine for Teichmüller space of a non-orientable surface with punctures] Let $n\geq 1$ and take $g+n>1$.
    There is a  $ \PMod(N_{g+1}^n)$-equivariant spine $\calZ(=\calZ_{g+1}^{n,1})$ for the Teichmüller space $\T(N_{g+1}^n)$    of dimension
    \[
    \dim(\calZ) = \begin{cases}
    2g+n-1  & \text{if } n>1, \\
    2g-1 & \text{if } n=1.
    \end{cases} 
\]
\end{corollary}




\section{Decorated Teichmüller spaces as classifying spaces for proper actions }\label{teic:model:proper:action} 
Given a discrete group $G$, a model for the {\it classifying space of $G$ for proper actions $\underbar{E}G$} is a $G$-CW-complex $X$  such that the fixed point set $X^H$ of a subgroup $H < G$ is contractible when $H$ is finite, and is empty otherwise. Such a model always exists and is unique up to $G$-homotopy.  Constructing explicit models for $\underline{E}G$ of `small' dimension is of particular interest, as they appear explicitly in the statement of the Baum-Connes conjecture and can be useful for computation; see for example \cite{LR05,luck2024isomorphism}.

The \emph{proper geometric dimension of $G$}, denoted by $\gdfin(G)$, is the minimum $n$ for which there exists an $n$-dimensional model for $\underline{E}G$. For any virtually torsion-free group $G$ it is well-known that $\vcd(G)\leq \gdfin(G)$. However, there are groups for which the inequality is strict, see for instance \cite{leary2003some,degrijse2017dimension,degrijse2017dimension}.

The Teichmüller space $\T(F_g^s)$ is known to be a model for the classifying space of $\Mod^{\pm}(F_g^s)$  for proper actions; see for instance \cite[Proposition 2.3]{MR2581835}.  We promote this statement to the non-orientable setting. Let $n\geq 0$ and $s=2n$.

\begin{proposition}
The Teichmüller space $\T(N_{g+1}^n)$ is a model of $\underline E\Mod(N_{g+1}^n)$ of dimension  $3g+2n-3$.
\end{proposition} 
\begin{proof}
By \cite[Lemma 5.8]{Lu05}, $\T(F_g^s)^\sigma$ is a model for $C_{\Mod^{\pm}(F_g^s)}(\sigma)$, the centralizer of $\sigma$ in $\Mod^{\pm}(F_g^s)$. Since $\Mod(N_{g+1}^n)$ can be realized as a subgroup of $C_{\Mod^{\pm}(F_g^s)}(\sigma)$, it follows that $\T(F_g^s)^\sigma$ is a model of $\underline E\Mod(N_{g+1}^n)$. By \cref{Teichmüller:orientable:nonorientable} there is a $\Mod(F_g^s)$-equivariant  homeomorphism between $\T(N_{g+1}^n)$ and $\left(\T(F_g^s)\right)^\sigma$, which concludes the proof.
\end{proof}

Now consider $1\leq \ell\leq n$ and take $m=2\ell$ and $n=2s$. Let $\T(N_{g+1}^n;\pi(\Delta))$ and $\T(F_g^s;\Delta)$ denote the decorated Teichmüller spaces of $N_{g+1}^n$ and its orientable double cover $F_g^s$, respectively, as considered in \cref{Embed:Teich:Non:Orient}. We now show that these spaces are also models of classifying spaces for proper actions.

\begin{proposition}
The decorated Teichmüller space $\T(F_g^s;\Delta)$ is a model of $\underline E\modeFD$ of dimension  ${6g+2s+m-7}$. In particular, $\T(F_g^s;\Delta)$ is a model of $\underline E\PMod^{\pm}(F_g^s).$ 
\end{proposition} 
\begin{proof}
    Since $\modeFD$ is a subgroup of $\Mod^{\pm}(F_g^s)$, then the Teichmüller space $\T(F_g^{s})$ is also a model of $\underline E\modeFD$.  On the other hand, recall from  \cref{Teichmüller:orientable:nonorientable} that the action of $\modeFD$ on the decorated Teichmüller space $\T(F_g^s;\Delta) \approx \T(F_g^{s}) \times {\rm K}^{m-1}$ is diagonal. The proposition then follows from noticing that $\modeFD$ acts on ${\rm K}^{m-1}$ simplicially, and hence ${\rm K}^{m-1}$ is equivariantly contractible.
\end{proof}

\begin{proposition}\label{thm:ModelDecorated}
The decorated Teichmüller space $\T(N_{g+1}^n;\pi(\Delta))$ is a model of the classifying space $\underline E\Mod(N_0,\pi(\Delta))$ of dimension  $3g+2n+\ell-4$. In particular, $\T(N_{g+1}^n;\pi(\Delta))$ is a model of $\underline E\PMod(N_{g+1}^n).$ 
\end{proposition}
\begin{proof}
Since $\T(F_g^s;\Delta)$ is a model for $\underline E\modeFD$, it follows from \cite[Lemma 5.8]{Lu05} that $\T(F_g^s;\Delta)^\sigma$ is a model for tha classifying space of proper actions for $C_{\modeFD}(\sigma)$, the centralizer of $[\sigma]$ in $\modeFD$. Since $\Mod(N_0,\pi(\Delta))$ can be realized as a subgroup of $C_{\modeFD}(\sigma)$, then $\T(F_g^s;\Delta)^\sigma$ is also a model for $\underline E\Mod(N_0,\pi(\Delta))$. The proposition follows since from \cref{prop:NonOrientDecorated} there is an equivariant homeomorphism between  $\T(N_{g+1}^{n};\pi(\Delta))$ and  $\T(F_g^s;\Delta)^\sigma$. 
\end{proof}

In \cite[Theorem~6.9]{I87} Ivanov computed the virtual cohomological dimension of $\modgsn$. For $n\geq 1$ and $g+n>1$, it is known to be $\vcd(\modgsn)=2g+n-2$. Therefore, the decorated Teichmuller spaces $\T(N_{g+1}^n;\pi(\Delta))$ give models for $\underline E\PMod(N_{g+1}^n)$ of dimensions greater than $\vcd(\PMod(N_{g+1}^n))$. For orientable surfaces, Harer's spine gives a model for $\underline E\PMod(F_{g}^s)$ of smaller dimension, which is actually of minimal dimension when $s=1$; see for example \cite[Introduction]{CJRLASS}. \cref{thm:ModelDecorated} combined with \cref{restriction:Retraction} and \cref{thm:main} imply the equivalent statement for non-orientable surfaces.

\begin{corollary}\label{spine:model:proper:action} Let $1\leq \ell\leq n$ and take $g+n>1$.
    The spine $\calZ=\calZ_{g+1}^{n,\ell}$ is a model for $\underline E \PMod(N_{g+1}^n)$ of dimension
    \[
    \dim(\calZ) = \begin{cases}
    \vcd(\PMod(N_{g+1}^n))+\ell  & \text{if } \ell< n, \\
    \vcd(\PMod(N_{g+1}^n))+n-1 & \text{if } \ell = n.
    \end{cases} 
\]
In particular, $\calZ_{g+1}^{1,1}$  gives a model of minimal dimension for $\underline E \Mod(N_{g+1}^1)$.
\end{corollary}

It follows that $\underline{\gd}(\PMod(N_{g+1}^1)) = \vcd(\PMod(N_{g+1}^1))$. Combining this with an argument using the Birman exact sequence as in the proof of \cite[Corollary 3.5]{CJRLASS} we can show the existence of a model for $\underline E \PMod(N_{g+1}^n)$ of minimal dimension when $n\geq 2$; this is \cref{proper:dim:non:orien:mcgs:1}.



\bibliographystyle{alpha} 
\bibliography{mybib}

\end{document}

%% file: Preamble.tex
\usepackage{graphicx}
\usepackage{pinlabel} 
\usepackage{hyperref}
\usepackage{amsmath}
\usepackage{amsfonts}
\usepackage{amssymb}
\usepackage{mathtools}
\usepackage[all,cmtip]{xy}
\usepackage{amsthm}
\usepackage{tikz-cd}
\usepackage{comment}
\usepackage{enumerate}
\usepackage{url}
\usepackage[utf8]{inputenc}
\usepackage[capitalize]{cleveref}
\usepackage{geometry}
\usepackage{mathrsfs}
\usepackage{xcolor}
\usepackage[font=small,labelfont=rm]{caption}
\usepackage[font=small,labelfont=rm]{subcaption}

\makeatletter
\providecommand\@dotsep{5}
\def\listtodoname{List of Todos}
\def\listoftodos{\@starttoc{tdo}\listtodoname}
\makeatother

\newtheorem{theorem}{Theorem}[section]
\newtheorem{proposition}[theorem]{Proposition}
\newtheorem{corollary}[theorem]{Corollary}
\newtheorem{lemma}[theorem]{Lemma}

\newcommand{\mycomment}[1]{}

  \theoremstyle{definition}
\newtheorem{definition}[theorem]{Definition}
\newtheorem{example}[theorem]{Example}

\newtheorem{remark}[theorem]{Remark}

\newtheorem{notation}[theorem]{Notation}

\newcommand{\R}{\mathbb{R}}

\newcommand{\calA}{{\mathcal A}}

\newcommand{\calY}{\mathcal Y}
\newcommand{\calZ}{\mathcal Z}

\newcommand{\asA}{{\mathbf A}} 
\newcommand{\asB}{{\mathbf B}} 

\newcommand{\T}{\mathscr T}

\newcommand{\nbeq}{\begin{equation}}
\newcommand{\neeq}{\end{equation}}
\newcommand{\beq}{\begin{equation*}}
\newcommand{\eeq}{\end{equation*}}

\DeclareMathOperator{\vcd}{vcd}

\DeclareMathOperator{\gd}{gd}
\DeclareMathOperator{\gdfin}{\underline{gd}}

\DeclareMathOperator{\Mod}{Mod}
\DeclareMathOperator{\PMod}{PMod}

\newcommand{\mode}{\Mod^{\pm}}
\newcommand{\modeFD}{\Mod^{\pm}(F_0;\Delta)}  
\newcommand{\modgs}[1]{\Mod({#1}_{g}^{s})}
\newcommand{\pmodgs}[1]{\PMod({#1}_{g}^{s})}

\newcommand{\modgso}{\Mod({F}_{g}^{s})}  

\newcommand{\modgsn}{\Mod({N}_{g+1}^{n})}  

\newcommand{\modgse}{\Mod^{\pm}({F}_{g}^{s})}

\DeclareMathOperator{\Diff}{Diff}

\DeclareMathOperator{\rank}{rank}